\documentclass[cjm]{ipart}
\usepackage{latexsym, amsmath,amssymb, hyperref}
\usepackage{graphicx}
\usepackage{caption}
\usepackage{subcaption}
\usepackage{amsthm}

\pubyear{2014}
\volume{0}
\issue{0}
\firstpage{1}
\lastpage{1}
\arxiv{1604.03189}

\startlocaldefs
\renewcommand{\epsilon}{\varepsilon}
\newcommand{\newsection}[1]
{\subsection{#1}\setcounter{theorem}{0} \setcounter{equation}{0}
\par\noindent}

\newtheorem{theorem}{Theorem}

\newtheorem{lemma}[theorem]{Lemma}
\newtheorem{corr}[theorem]{Corollary}

\newtheorem{proposition}[theorem]{Proposition}
\newtheorem{deff}[theorem]{Definition}

\newcommand{\bth}{\begin{theorem}}
\newcommand{\ble}{\begin{lemma}}
\newcommand{\bcor}{\begin{corr}}

\newcommand{\bdeff}{\begin{deff}}

\newcommand{\bprop}{\begin{proposition}}
\newcommand{\ele}{\end{lemma}}
\newcommand{\ecor}{\end{corr}}
\newcommand{\edeff}{\end{deff}}

\newcommand{\eprop}{\end{proposition}}

\newcommand{\la}{\lambda}

\newcommand{\e}{\varepsilon}

\renewcommand{\Pi}{\varPi}

\renewcommand{\epsilon}{\varepsilon}

\newcommand{\Rt}{{\Bbb R}^2}

\newcommand{\tidle}{\tilde}

\newcommand{\R}{{\mathbb R}}
\newcommand{\1}{{\rm 1\hspace*{-0.4ex}%
\rule{0.1ex}{1.52ex}\hspace*{0.2ex}}}
\newcommand{\tg}{\tilde g}
\newcommand{\dgt}{d_{\tilde g}}
\newcommand{\sqrtg}{\sqrt{-\Delta_g}}
\newcommand{\sqrtd}{\sqrt{-\Delta_{\tilde g}}}
\newcommand{\Pe}{\sqrt{-\Delta}}

\endlocaldefs

\begin{document}

\begin{frontmatter}

\title{Geodesic period integrals of eigenfunctions on  Riemannian 
surfaces and the Gauss-Bonnet Theorem\thanksref{t1}}
\thankstext{t1}{The authors were supported in part by the NSF grant DMS-1361476}
\runtitle{Geodesic period integrals of eigenfunctions}


\author{\fnms{Christopher D.} \snm{Sogge}\ead[label=e1]{sogge@jhu.edu}}
\address{Department of Mathematics\\ Johns Hopkins University\\ Baltimore, MD 21218\\\printead{e1}}
\author{\fnms{Yakun} \snm{Xi}\ead[label=e2]{ykxi@math.jhu.edu}}
\address{Department of Mathematics\\ Johns Hopkins University\\ Baltimore, MD 21218\\\printead{e2}}
\and
\author{\fnms{Cheng} \snm{Zhang}\ead[label=e3]{czhang67@jhu.edu}}
\address{Department of Mathematics\\ Johns Hopkins University\\ Baltimore, MD 21218\\\printead{e3}}

\runauthor{C. D. Sogge, Y. Xi and C. Zhang}

\begin{abstract}
We use the Gauss-Bonnet theorem and the triangle comparison theorems of Rauch and Toponogov to show that on compact Riemannian surfaces of negative curvature
period integrals of eigenfunctions  $e_\la$ over geodesics go to zero
at the rate of $O((\log\la)^{-1/2})$ if $\la$  are their frequencies.  As discussed in \cite{CSPer}, no such result is possible in the constant curvature case if the curvature is $\ge0$.  Notwithstanding, we also show that these bounds for period integrals are valid provided that  integrals of the curvature over all geodesic balls of radius
$r\le 1$ are pinched from above by $-\delta r^N$ for some fixed $N$ and $\delta>0$.  This allows, for instance, the curvature to be nonpositive and to vanish
of finite order at a finite number of isolated points. Naturally, the above results also hold for the appropriate type of quasi-modes.
\end{abstract}

\received{\smonth{4} \sday{19}, \syear{2016}}

\begin{keyword}[class=AMS]
\kwd[Primary ]{35F99;}
\kwd[ Secondary ]{35L20,} {42C99}
\end{keyword}

\begin{keyword}
\kwd{Eigenfunction}
\kwd{negative curvature}
\end{keyword}
\end{frontmatter}
\newsection{Introduction and preliminaries}

Using Kuznecov formulae, Good~\cite{Good} and Hejhal~\cite{Hej} showed that if $\gamma_{per}$ is a periodic geodesic on a compact hyperbolic
surface $M$ then
\begin{equation}\label{i.1}
\Bigl|\, \int_{\gamma_{per}} e_\la \, ds\, \Bigr| \le C_{\gamma_{per}},
\end{equation}
with $ds$ denoting arc length measure on $\gamma_{per}$ and with $e_\la$ denoting the $L^2$-normalized eigenfunction on $M$, i.e., 
$$-\Delta_g e_\la=\la^2 e_\la, \quad \text{and } \, \, \int_M |e_\la|^2 \, dV_g=1.$$
Here $\Delta_g$ denotes the Laplace-Beltrami operator on $(M,g)$ and $dV_g$  is the volume element.

This result was generalized by Zelditch~\cite{ZelK}, who showed that if $\la_j$ are the eigenvalues of $\sqrt{-\Delta_g}$  on
an compact Riemannian surface and if $a_j(\gamma_{per})$ denote the period integrals in \eqref{i.1} 
for an orthonormal basis of eigenfunctions with eigenvalues $\la_j$
then
$$\sum_{\la_j\le \la}|a_j(\gamma_{per})|^2 =c_{\gamma_{per}} \la +O(1),$$
which implies \eqref{i.1}.  Further work for hyperbolic surfaces giving more information about the lower order terms in terms
of geometric data for $\gamma_{per}$ was done by Pitt~\cite{Pitt}.  Since the number of eigenvalues that are smaller than
$\la$ is $O(\la^2)$, this asymptotic formula implies that, on average, one can do much better than \eqref{i.1}.  The problem
of improving this upper bound was raised and discussed in Pitt~\cite{Pitt} and Reznikov~\cite{Rez}.

In an earlier joint paper of Chen and the first author~\cite{CSPer}, it was pointed out that no improvement of \eqref{i.1} is possible on compact
two-dimensional manifolds of constant non-negative curvature.  For instance, on $S^2$, the integrals in \eqref{i.1} have unit size if
$\gamma_{per}$ is the equator and $e_\la$ is an $L^2$-normalized zonal function of even degree.  Also on ${\mathbb T}^2$, for every
periodic geodesic, $\gamma_{per}$, one can find a sequence of eigenvalues $\la_k$ and eigenfunctions $e_{\la_k}$ so that
$e_{\la_k}\equiv 1$ on $\gamma_{per}$ and $\|e_{\la_k}\|_{L^2({\mathbb T}^2)}\approx 1$.  

Despite this, in \cite{CSPer}, it was shown that the period integrals in \eqref{i.1} are $o(1)$ as $\la\to \infty$ if $(M,g)$ has strictly negative
curvature.  The proof exploited the fact that, in this case, quadrilaterals always have their four interior angles summing to a value
strictly smaller than $2\pi$.   This ``defect'' (see Figure~\ref{fig1})  allowed the authors to obtain $o(1)$ decay for period integrals using a stationary
phase argument involving reproducing kernels for the eigenfunctions.

The purpose of this paper is to improve this result in two ways.  First, even though there can be no decay for period integrals for the flat
two-torus, we shall obtain decay if the curvature $K=K_g$ of $(M,g)$ is assumed to be non-positive but allowed to vanish at an averaged
rate of finite type in the sense
that whenever $B_r\subset M$ is a geodesic ball of radius $r\le1$ (and arbitrary center)  we have that
\begin{equation}\label{i.2}
\int_{B_r}K\, dV_g \le -\delta r^N, \quad r\le 1,
\end{equation}
for some fixed $\delta>0$ and $N<\infty$.  Of course if $K\le -\delta$ everywhere then we can take $N=2$ in \eqref{i.2}
(and possibly have to replace $\delta$ by a multiple of itself).  Condition \eqref{i.2} holds, for instance, if the curvature is negative off of a finite
collection of points where it vanishes to finite order.  Besides this improvement, we shall also show that, under the assumption \eqref{i.2}, the period
integrals in \eqref{i.1} are $O((\log\la)^{-1/2})$.

To be more specific, our main result is the following.

\begin{theorem}\label{mainthm}
Let $(M,g)$ be a compact two-dimensional boundaryless manifold.  Assume that its curvature satisfies \eqref{i.2}.  Then if $\gamma(t)$ is a geodesic
in $M$ parametrized by arc length and if $b\in C^\infty_0((-1/2,1/2))$ we have for $\la\gg 1$
\begin{equation}\label{i.3}
\Bigl|\, \int b(t)\,e_\la(\gamma(t))\,  dt \, \Bigr|\le C_{M,b}(\log \la)^{-1/2},
\end{equation}
where $C_{M,b}$ depends on $M$ and $b$, but not on $\gamma$.  Additionally, if $\gamma_{per}$ is a periodic geodesic and if $|\gamma_{per}|$ denotes
its length then for $\la\gg 1$
\begin{equation}\label{i.4}
\Bigl| \, \int_{\gamma_{per}} e_\la \, ds \, \Bigr|\le C_M \, |\gamma_{per}|\, (\log\la)^{-1/2},
\end{equation}
where
$C_M$ depends only on $(M,g)$.
\end{theorem}
 
 If one uses a partition of unity argument, it is clear that \eqref{i.3} implies \eqref{i.4}.  So we only need to prove the former.
 
 The broad strategy will be similar to the earlier work of Chen and the first author~\cite{CSPer}.  We shall need to refine the stationary phase
 arguments used there a bit and use the Gauss-Bonnet theorem to exploit the aforementioned ``defects'' of quadrilaterals that arise in these
 arguments, which allow one to obtain favorable control of lower bounds for first and second derivatives of the phase functions 
 occurring in the stationary phase arguments (unlike in the case of the two-torus).

This paper is organized as follows.  In the next section we shall show that we can prove \eqref{i.3} by estimating
integrals over geodesics in the universal cover of $(M,g)$ that arise from reproducing kernels for eigenfunctions.
We shall also see here that \eqref{i.3} also holds when the eigenfunctions are replaced by appropriate types of 
quasi-modes.  In \S 3, using the Gauss-Bonnet theorem and triangle comparison theorems, we shall collect
the geometric facts that we shall need for our estimates.  
In \S 4 we shall derive some simple one-dimensional stationary phase estimates that will be needed for our proof.
In the next section, we shall use the Hadamard parametrix
to show that the oscillatory integrals that we need to estimate lend themselves to  these stationary phase estimates.  We
shall also show that we can get favorable bounds for first and second derivatives of the phase functions using the
aforementioned geometric facts.  In the final section we put things together and finish the proof of our main
estimate \eqref{i.3}.

In what follows, as we may, we shall assume that the injectivity radius of $(M,g)$ is ten or more and that
its nonpositive curvature is pinched below by $-1$, i.e., $-1\le K\le 0$.

 \newsection{Hadamard's theorem and a standard reduction}
 
To prove \eqref{i.3} let us first fix a real-valued function $\rho\in {\mathcal S}(\R)$ satisfying
$$\rho(0)=1 \quad \text{and } \, \, \Hat \rho(\tau)=0 , \quad |\tau|\ge 1/4.$$
Then since $\rho(T(\la-\sqrt{-\Delta_g}))e_\la=e_\la$, for any $T>0$, in order to prove \eqref{i.3} it suffices
to show that we can choose $T=T(\la)$ so that for $\la\gg 1$ we have the uniform bounds
\begin{equation}\label{2.1}
\Bigl|\, \int b(t) \bigl(\rho(T(\la-\sqrt{-\Delta_g}))f\bigr)(\gamma(t)) \, dt \, \Bigr|\le C_{M,b}\, (\log\la)^{-1/2}\, \|f\|_{L^2(M)}.
\end{equation}
To do this we shall take
\begin{equation}\label{2.2}
T=c\log\la,
\end{equation}
where $c=c_M>0$ is a small constant depending on $(M,g)$.

Let $\{e_j\}$ be an orthonormal basis of eigenfunctions with eigenvalues $\{\la_j\}$, and let
$$E_jf=\langle f,e_j\rangle e_j,$$
denote the projection of $f\in L^2(M)$ onto the eigenspace with eigenvalue $\la_j$.  
Then since $\rho(\tau)\ge 1/2$ for $|\tau|\le \delta$, some
$\delta>0$, clearly \eqref{2.1}-\eqref{2.2} imply that 
\begin{multline}\label{2.3}
\Bigl|\, \int b(t) \bigl(\rho(T(\la-\sqrt{-\Delta_g})) \chi_{[\la-(\log\la)^{-1},\la+(\log\la)^{-1}]}f\bigr)(\gamma(t)) \, dt \, \Bigr|
\\
\le C_{M,b}\, (\log\la)^{-1/2}\, \|f\|_{L^2(M)},
\end{multline}
if
$$ \chi_{[\la-(\log\la)^{-1},\la+(\log\la)^{-1}]}f=\sum_{|\la-\la_j|\le (\log\la)^{-1}}E_j f$$
denotes the projection of $f$ onto a spectral band of width $(\log\la)^{-1}$ about $\la$.
Using standard arguments (see \cite{SZQuas}) one sees from this that we have
\begin{equation}\label{2.4}
\Bigl|\, \int b(t) \, \Psi_\la(\gamma(t)) \, dt \, \Bigr| \le C_{b,M} (\log\la)^{-1/2},
\end{equation}
for quasi-modes $\Psi_\la$ satisfying
\begin{equation}\label{2.5}
(\log\la/\la)\|(\Delta_g+\la^2)\Psi_\la\|_{L^2(M)}+\|\Psi_\la\|_{L^2(M)}\le 1
\end{equation}
with $\la \gg 1$.
Of course \eqref{2.4} implies that when \eqref{2.5} holds we also have the following analog of \eqref{i.4}
\begin{equation}\label{2.6}
\Bigl| \, \int_{\gamma_{per}}\Psi_\la \, ds \, \Bigr|\le C_M|\gamma_{per}| \, (\log\la)^{-1/2}
\end{equation}
if $\gamma_{per}$ is a periodic geodesic in $M$.

To set up the proof of \eqref{2.1} we first note that the kernel of the operator there is given by
$$\rho\bigl(T(\la-\sqrt{-\Delta_g})\bigr)(x,y)=\sum_j \rho(T(\la-\la_j)) e_j(x)\overline{e_j(y)}.$$
By Schwarz's inequality, we would have \eqref{2.1} if we could show that
\begin{multline*}\int_M\Bigl|\, \int b(t) \sum_j \rho\bigl(T(\la-\la_j)\bigr) \, e_j(\gamma(t)) \overline{e_j(y)} \, dt \, \Bigr|^2 \, dV_g(y)
\le C_{b,M}(\log\la)^{-1}, 
\\
 \la\gg 1.
\end{multline*}
By orthogonality, if $\chi(\tau)=(\rho(\tau))^2$, this is equivalent to showing that if
$$b(t,s)=b(t)b(s)\in C^\infty_0((-1/2,1/2)^2),$$
then
\begin{equation}\label{2.7}
\Bigl|\, \iint b(t,s) \sum_j \chi(T(\la-\la_j)) e_j(\gamma(t)) \overline{e_j(\gamma(s))} \, dt ds\, \Bigr|
\le C_{b,M}(\log \la)^{-1},
\end{equation}
if $\la \gg 1$.

Note that
$$\sum_j \chi(T(\la-\la_j))e_j(x)\overline{e_j(y)} =\frac1{2\pi T} \int \Hat \chi(\tau/T) e^{i\tau \la} 
\bigl(e^{-i\tau \sqrt{-\Delta_g}}\bigr)(x,y) \, d\tau.$$
As a first step in the proof of \eqref{2.7} fix  a bump function $\beta\in C^\infty_0(\R)$ satisfying
$$\beta(\tau)=1, \, \, \, |\tau|\le 3 \quad \text{and } \, \, \beta(\tau)=0, \, \, \, |\tau|\ge 4.$$
Then the proof of Lemma 5.1.3 in \cite{SFIO} shows that, because of our assumption that the injectivity radius of
$(M,g)$ is ten or more, we can write
\begin{multline}\label{2.8}
\frac1{2\pi T} \int \beta(\tau) \Hat \chi(\tau/T) e^{i\tau \la} \bigl(e^{-i\tau\sqrt{-\Delta_g}}\bigr)(x,y) \, d\tau
\\
=\frac{\la^{1/2}}T \sum_\pm a_\pm(\la; d_g(x,y))e^{\pm i \la d_g(x,y)} +O(1/T),
\end{multline}
if $d_g$ denotes the Riemannian distance on $(M,g)$, where
\begin{equation}\label{2.9}
\Bigl|\, \frac{d^j}{dr^j} a_\pm(\la;r)\, \Bigr|\le C_j r^{-j-1/2} \quad \text{if } \, \, r\ge \la^{-1},
\end{equation}
and
\begin{equation}\label{2.10}
|a_\pm(\la;r)|\le C\la^{1/2} \quad \text{if } \, \, \, 0\le r\le \la^{-1}.
\end{equation}

Since $d_g(\gamma(t),\gamma(s))=|t-s|$, we conclude from \eqref{2.8} that we would have, for a given $c>0$,
\begin{multline}\label{2.11}
\frac1{2\pi T}\Bigl|\, \iiint b(t,s) \beta(\tau)\Hat \chi(\tau/T) e^{i\tau\la} \bigl(e^{-i\tau\sqrt{-\Delta_g}}\bigr)(\gamma(t),\gamma(s)) \, d\tau dt ds\, \Bigr|
\\
\le C_{b,M}(\log\la)^{-1}, \quad \text{if } \, \, T=c\log\la,
\end{multline}
if
$$\la^{1/2}\Bigl|\, \iint b(t,s)e^{\pm i\la|t-s|} a_\pm(\la;|t-s|) \, dt ds \, \Bigr|\le C_{b,M}.$$
Since the latter estimate is a simple consequence of \eqref{2.9} and \eqref{2.10}, we obtain \eqref{2.11}.

In view of \eqref{2.11}, we conclude that we would have \eqref{2.7} if we could obtain the following bounds for
the remaining part of $\chi(T(\la-\sqrt{-\Delta_g}))$:
\begin{multline*}\label{2.12}
\frac1{2\pi T}\Bigl| \iiint b(t,s)\bigl(1-\beta(\tau)\bigr) \, \Hat \chi(\tau/T) e^{i\tau\la}
\bigl(e^{-i\tau\sqrt{-\Delta_g}}\bigr)(\gamma(t),\gamma(s)) \, d\tau ds dt \, \Bigr|
\\
\le C_{b,M}(\log\la)^{-1},
\end{multline*}
if $T$ is as in \eqref{2.2}.
Note that for $T\ge 1$ we have the uniform bounds
$$\frac1{2\pi T}\Bigl| \, \int \bigl(1-\beta(\tau)\bigr) \Hat \chi(\tau/T) e^{i\tau (\la+\la_j)} \, d\tau\, \Bigr|
\le C_N(1+|\la+\la_j|)^{-N}, \, \, \, N=1,2,\dots,$$
and so, since $\la_j\ge 0$ and $\la\gg 1$,

\begin{multline} \frac1{2\pi T}\Bigl| \, \int \bigl(1-\beta(\tau)\bigr) \Hat \chi(\tau/T) e^{i\tau \la}
\bigl(e^{i\tau \sqrt{-\Delta_g}}\bigr)(\gamma(t),\gamma(s)) \, d\tau\, \Bigr| 
\\
\le C_N(1+\la)^{-N}, \, \, \, N=1,2,\dots .\end{multline}

Thus, by Euler's formula, to prove \eqref{2.12}, it suffices to show that if $T$ is as in \eqref{2.2} (for an appropriate choice of $c=c_M>0$) we have
\begin{multline}\label{2.13}
\Bigl| \, \iiint b(t,s) \bigl(1-\beta(\tau)\bigr)\Hat \chi(\tau/T) e^{i\tau \la} \bigl(\cos \tau \sqrt{-\Delta_g}\bigr)(\gamma(t),\gamma(s)) \, d\tau dt ds \, \Bigr|
\\
\le C_{b,M}.
\end{multline}
Here $\bigl(\cos \tau\sqrt{-\Delta_g}\bigr)(x,y)$ is the kernel for the map $C^\infty(M)\ni f\to u\in C^\infty(\R\times M)$ solving the Cauchy problem
with initial data $(f,0)$, i.e.,
$$\bigl(\partial_\tau^2-\Delta_g\bigr)u=0, \quad u(0, \, \cdot \, )=f, \quad \partial_\tau u(0, \, \cdot \, )=0.$$

To be able to compute the integral in \eqref{2.13} we need to relate this wave kernel to the corresponding one in the universal cover for $(M,g)$.  Recall that by a theorem of Hadamard (see \cite[Chapter 7]{doCarmo}) for every point $P\in M$, the exponential map at $P$, $\exp_P: T_PM\to M$ is a covering map.
We might as well take $P=\gamma(0)$ to be the midpoint of the 
geodesic segment$\{\gamma(t): \, |t|\le \tfrac12\}$.  If we identify $T_PM$ with $\Rt$, and let $\kappa$ denote this exponential
map then $\kappa: \Rt\to M$ is a covering map.  We also will denote by $\tilde g$ the metric on ${\mathbb R}^2$ which is the pullback via $\kappa$ of the the metric
$g$ on $M$.  Also, let $\Gamma$ denote the group of deck transformations, which are the diffeomorphisms $\alpha$ from $\Rt$ to itself preserving $\kappa$, i.e., $\kappa = \kappa \circ \alpha$.  
Next, let
$$D_{Dir}=\{ \tilde y\in \Rt: \, d_{\tg}(0,\tilde y)<\dgt(0, \alpha(\tilde y)), \, \forall \alpha \in \Gamma, \, \, \alpha \ne Identity\}$$
be the Dirichlet domain for $(\Rt,\tg)$, where $\dgt(\, \cdot \, ,\, \cdot \, )$ denotes the Riemannian distance function for $\Rt$ corresponding to
the metric $\tg$.   We can then add to $D_{Dir}$ a subset of $\partial D_{Dir}=\overline{D_{Dir}}\backslash
\text{Int }(D_{Dir})$ to obtain a natural fundamental domain $D$, which has the property that $\Rt$ is the disjoint union of the
$\alpha(D)$ as $\alpha$ ranges over $\Gamma$ and $\{\tilde y\in \Rt: \, \dgt(0,\tilde y)<10\} \subset D$ since we are assuming that the injectivity
radius of $(M,g)$ is more than ten.  It then follows that we can identify every point $x\in M$ with the unique point $\tilde x\in D$ having the property
that $\kappa(\tilde x)=x$.  Let also $\tilde \gamma(t)$, $|t|\le \tfrac12$ similarly denote those points in $D$ corresponding to our geodesic segment
$\gamma(t)$, $|t|\le \tfrac12$ in $M$.  Then $\{\tilde \gamma(t): |t|\le \tfrac12\}$ is a line segment of unit length whose midpoint is the origin, and we shall denote just by
$\tilde \gamma$ the line through the origin containing this segment.  Note that $\tilde \gamma$ then is a geodesic in $\Rt$ for the metric $\tg$, and 
the Riemannian distance between two points on $\tilde \gamma$ agrees with their Euclidean distance.
Finally, if $\Delta_{\tg}$ denotes the Laplace-Beltrami operator associated to $\tg$ then since solutions of the above Cauchy problem for $(M,g)$ correspond exactly
to periodic (i.e. $\Gamma$-invariant) solutions of the corresponding Cauchy problem associated to $\partial^2_t-\Delta_{\tg}$, we have the following
important formula relating the wave kernel on $(M,g)$ to the one for the universal cover $(\Rt,\tg)$:
\begin{equation}\label{2.14}
\bigl(\cos \tau \sqrtg \big)(x,y)=\sum_{\alpha\in \Gamma}\bigl(\cos \tau \sqrtd  \bigr)(\tilde x, \alpha(\tilde y)).
\end{equation}

Due to this formula, we would have \eqref{2.13} if we could show that for $T$ as in \eqref{2.2},
\begin{multline}\label{2.15}
\sum_{\alpha\in \Gamma} \Bigl| \, \iiint b(t,s) \bigl(1-\beta(\tau)\bigr) \Hat \chi(\tau/T) e^{i\tau \la}
\bigl(\cos \tau \sqrtd \bigr)(\tilde \gamma(t),\alpha(\tilde \gamma(s))) \, d\tau dt ds \, \Bigr|
\\
\le C_{b,M}.
\end{multline}

By Huygens principle, $\bigl(\cos \tau \sqrtd \bigr)(\tilde x,\tilde y)=0$ if $d_{\tilde g}(\tilde x,\tilde y)>\tau$, where $d_{\tilde g}$ denotes the
Riemannian distance on $(\Rt,\tilde g)$.  Since $\chi=\rho^2$ our assumption that $\rho(\tau)=0$ for $|\tau|\ge 1/4$ means that the integrand in
\eqref{2.15} vanishes when $|\tau|\ge T/2$.  Therefore, since there are $O(\exp(C_MT))$ ``translates'' of $D$ satisfying $d_{\tilde g}(D,\alpha(D))<T$,
we conclude that the sum in \eqref{2.15} involves $O(\exp(C_MT))$ nonzero terms.  Based on this, we conclude that we would have
\eqref{2.15} if we could prove the following.

\begin{proposition}\label{prop2.1}  Given our $(M,g)$ sastisfying \eqref{i.2} we can fix $c=c_M>0$ so that we have for $\la\gg 1$
\begin{multline}\label{2.16}
 \Bigl| \, \iiint b(t,s) \bigl(1-\beta(\tau)\bigr) \Hat \chi(\tau/T) e^{i\tau \la}
\bigl(\cos \tau \sqrtd \bigr)(\tilde \gamma(t),\alpha(\tilde \gamma(s))) \, d\tau dt ds \, \Bigr|
\\
\le C_{b,M} \la^{-\delta_M} \quad \text{if } \, \, \, T=c\log \la,
\end{multline}
for some $\delta_M>0$ which depends on $M$ but not on $b$ or $\la$.
\end{proposition} 

The power $\delta_M$ in \eqref{2.16}  depends on the power $N$ in our assumption \eqref{i.2}.  As we shall see we can
take it to be $1/10N$.
 
 \newsection{Geometric tools}
 
 In this section we are working with ${\mathbb R}^2$ equipped with the metric $\tilde g$ which is the pullback of 
the metric $g$ on $M$ via the covering map.  Thus, if $K$ denotes the Gaussian curvature on $({\mathbb R}^2,\tilde g)$ and 
if $B_r(\tilde x)$ denotes a geodesic ball of radius $r$ centered at some $\tilde x\in {\mathbb R}^2$, our curvature assumption \eqref{i.2} on
$(M,g)$ lifts to 
\begin{multline}\label{g.1}
\int_{B_r(\tilde x)}K\, dV \le - \delta r^N, \quad \text{if } \, r<1,     \, \ \text{and } \, \,  \tilde x\in {\mathbb R}^2,  
\\
\text{for some } \, \delta>0 \, \, \text{and } \, \,  N\ge2.
\end{multline}

To prove our estimates for period integrals over geodesics we shall require a couple consequences of elementary results from
Riemannian geometry.  One is based on \eqref{g.1} and the Gauss-Bonnet theorem.
As we pointed this assumption is valid when the curvature on $M$ is pinched from above by a negative  constant but allows 
situations where the curvature is nonpositive and vanishes on lower dimensional sets.  The other result is based on Togonogov's theorem 
and the fact that we are assuming that the curvature on $(M,g)$ and hence on $(\Rt,\tilde g)$ is pinched below by $-1$.

Let us now state the two geometric results that will play a key role in our analysis.

\begin{proposition}\label{propg.1}  Let $\tilde \gamma_1(t)$ and $\tilde \gamma_2(s)$, $|s|, \, |t|\le 1/2$ be two  unit length geodesics in $({\mathbb R}^2,\tilde g)$ parameterized by arc length
satisfying $d_{\tilde g}(\tidle \gamma_1(t),\, \tilde \gamma_2(s))\ge 1$, $|t|, \, |s| \le 1/2$.  Suppose that there is a $(t_0,s_0)\in [-1/2,1/2]\times [-1/2,1/2]$ so that the geodesic through $\tilde \gamma_1(t_0)$ and 
$\tilde \gamma_2(s_0)$ intersects $\tilde \gamma_1$  with angle $\theta_{t_0}$ and $\tilde \gamma_2$  with angle
$\theta_{s_0}$ (see Figure~\ref{fig0}) and suppose further that
\begin{equation}\label{g.2}
\theta_{t_0},\theta_{s_0} \in [\pi/2- \lambda^{-1/3}, \, \pi/2].
\end{equation}
Then if 
\begin{equation}\label{g.3}
\e_0=1/5N
\end{equation}
where $N$ is as in \eqref{g.1}
and if $\lambda$ is larger than a fixed constant
\begin{multline}\label{g.4}
\max \bigl(\pi/2-\theta_t, \, \pi/2-\theta_s \bigr) \ge \la^{-1/4}, 
\\
 \text{if } \, \, t,s\in [-1/2,1/2] \, \, \text{and } \,  \max\bigl(|t-t_0|,\, |s-s_0|\bigr)\ge \la^{-\e_0},
\end{multline}
if $\theta_t$ denotes the intersection angle of  $\tilde \gamma_1$ and the geodesic through $\tilde \gamma_1(t)$ and $\tilde \gamma_2(s)$ and $\theta_s$ denotes
the intersection angle of this geodesic and $\tilde \gamma_2$.
\end{proposition}

\begin{figure}[!ht]
  \centering
    \includegraphics[width=0.9\textwidth]{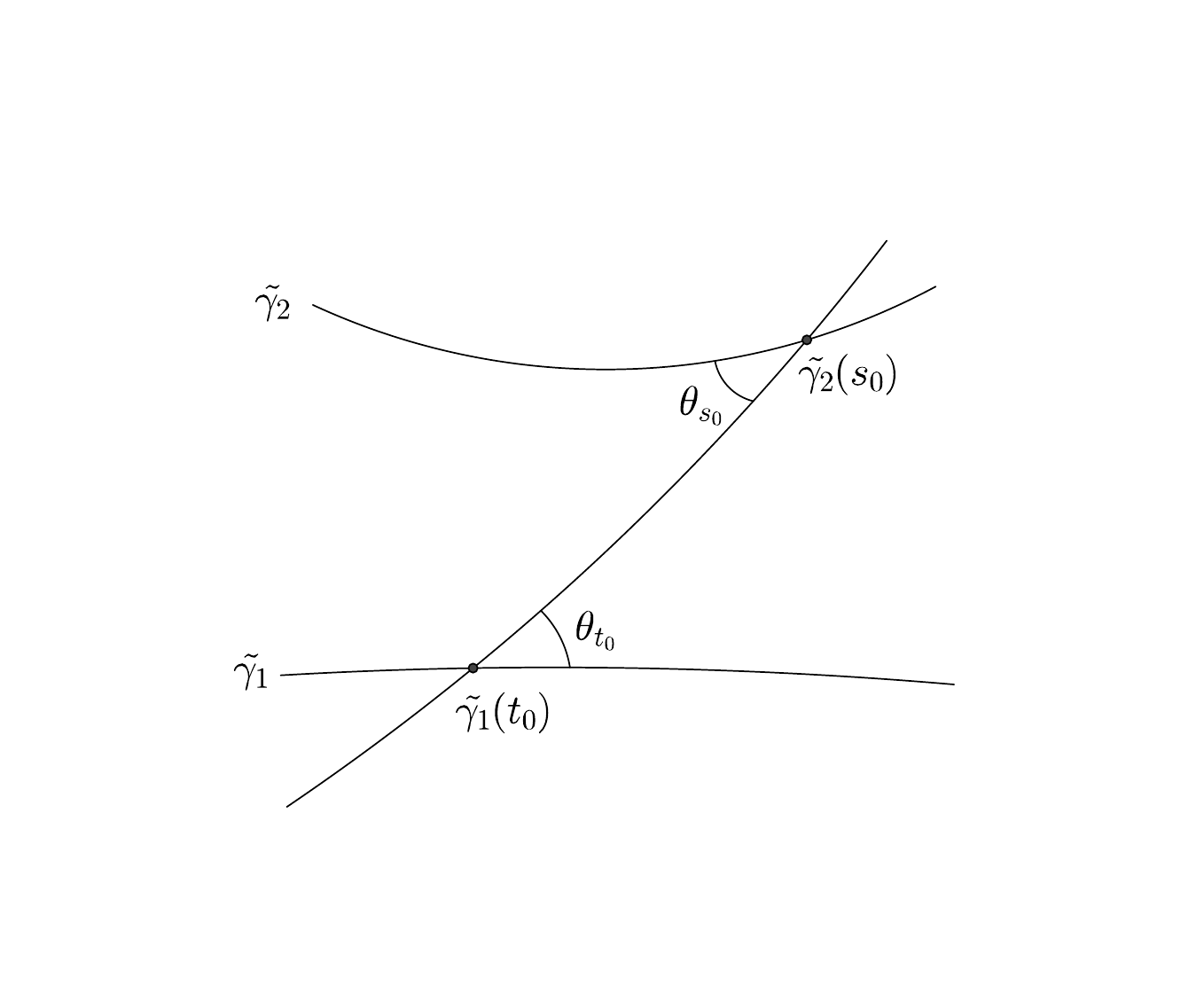}
      \caption{}
      \label{fig0}
\end{figure}

The other proposition that we need is the following simple consequence of Toponogov's theorem
which was used in earlier joint work of the first author and Blair \cite{BSTop}.

\begin{proposition} \label{topprop}  
As above assume that the Gaussian curvature of $({\mathbb R}^2,\tilde g)$ satisfies
$$K\ge -1.$$
Let $\tilde \gamma(t)$, $t\in \R$, be a geodesic with $\tilde \gamma(0)=P_0$.  Given $T\gg 1$,
let $C(\theta;T)$, $\theta\ll 1$, denote the set of points $Q\in B_T(P_0)$ which lie on a geodesic
though $P_0$ which intersects $\tilde \gamma$ of angle $\le \theta$.
Thus, $C(\theta;T)$ is the intersection of the geodesic ball $B_T(P_0)$ 
of radius $T$ about $P_0$
with the  cone of aperture $\theta$ about $\tilde \gamma$ with
vertex $P_0$.
  Then if $0<r\le 1$ and, if
$${\mathcal T}_r(\tilde \gamma)=\{x\in {\mathbb R}^2: \, d_{\tilde g}(x,\tilde \gamma)\le r\}$$
denotes the tube of radius $r$ about $\tilde \gamma$, we have that 
\begin{equation}\label{g.5}
C(\theta_{T,r};T)\subset {\mathcal T}_r(\tilde \gamma), \quad \text{if } \, \, \sin \tfrac12\theta_{T,r} = \frac{\sinh \tfrac12r}{\sinh T}, 
\quad \text{if } \, T>0.
\end{equation}
\end{proposition}

To prove Proposition~\ref{propg.1} we shall use a couple of special cases for the Gauss-Bonnet theorem (see \cite{doCarmo})  concerning the sum of the interior angles
$\alpha_j$ for geodesic quadrilaterals $Q$ and geodesic triangles ${\mathcal T}$ in $({\mathbb R}^2,g)$.  In the first case we define the ``defect'' of $Q$, $\text{Defect }Q$,
to be $2\pi$ minus the sum of the four interior angles at the vertices, and in the case of ${\mathcal T}$, we define  $\text{Defect }{\mathcal T}$ to be $\pi$ minus the sum of its
three interior angles, as shown in Figure~\ref{fig1}.

Then, by the Gauss-Bonnet theorem we have
\begin{gather*}
\text{Defect }Q=-\int_Q K\, dV
\\
\text{Defect }{\mathcal T}=-\int_{\mathcal T} K\, dV.
\end{gather*}

\begin{figure}[!ht]
  \centering
    \makebox[\textwidth][c]{\includegraphics[width=1.2\textwidth]{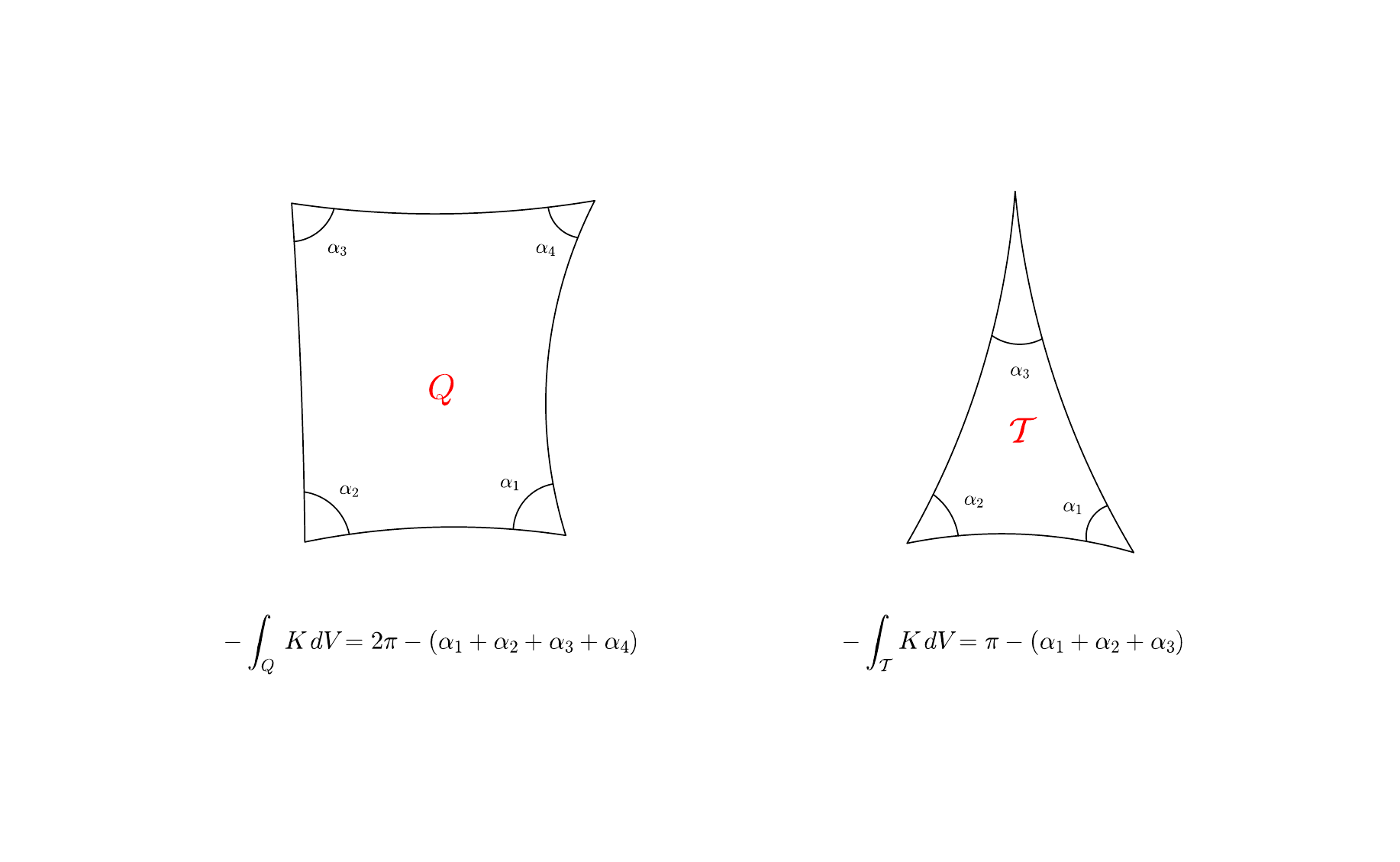}}
      \caption{Gauss-Bonnet Theorem}
      \label{fig1}
\end{figure}

\begin{proof}[Proof of Proposition~\ref{propg.1}]
Suppose that,  for a given $(t_0,s_0)\in [-1/2,1/2]\times [-1/2,1/2]$, \eqref{g.2} is valid.  By symmetry it suffices to show
that the conclusion in \eqref{g.4} is valid if we assume that $t\in [-1/2,1/2]\, \backslash \, (t_0-\la^{-\e_0},\, t_0+\la^{-\e_0})$ and
$|s|\le 1/2$.  If $s\ne s_0$ there are two cases as shown in Figure~\ref{fig2}:  Either the geodesic segment connecting 
$\tilde \gamma_1(t_0)$ and $\tilde \gamma_2(s_0)$ and the one connecting $\tilde \gamma_1(t)$ and $\tilde \gamma_2(s)$ do not intersect
or intersect.  In the first case we obtain a geodesic quadrilateral $Q$ with vertices $\tilde \gamma_1(t_0), \, \tilde \gamma_1(t), \, \tilde\gamma_2(s_0)$
and $\tilde\gamma_2(s)$, while in the other case we obtain two geodesic triangles using those four points  and the intersection
point of the aforementioned geodesic segments.  To reach this conclusion we are using the fact that since we are assuming $K\le 0$,
two geodesics in $({\mathbb R}^2,\tilde g)$ are disjoint or intersect at exactly one point by the Cartan-Hadamard theorem.

In the first case, let $\alpha_{t_0}, \alpha_t,\alpha_{s_0}$ and $\alpha_s$ denote the interior angles of the geodesic quadrilateral $Q$
at vertices $\tilde \gamma_1(t_0), \, \tilde \gamma_1(t), \, \tilde \gamma_2(s_0)$ and $\tilde \gamma_2(s)$, respectively.  Note that $\alpha_t=\theta_t$
if $0<\alpha_t\le \pi/2$ and $\theta_t=\pi-\alpha_t$ if $\alpha_t\in (\pi/2,\pi)$, etc.

As we mentioned before, by the Gauss-Bonnet theorem
\begin{equation}\label{g.6}
\text{Defect }Q=2\pi -\bigl(\alpha_{t_0}+\alpha_t+\alpha_{s_0}+\alpha_s\bigr)=-\int_QK\, dV.
\end{equation}
As in Figure~\ref{fig2}, if we consider the geodesic ball, $B_r$, $r=\la^{-\e_0}/100$, which is tangent to $\tilde \gamma_1$ at $\tilde \gamma_1((t+t_0)/2)$ and
on the same side of $\tilde \gamma_1$ as $Q$, it follows that, if $\la$ is larger than a fixed constant depending on the metric, we have
$B_r\subset Q$ if $\alpha_t\notin (0,\pi/4)\cup (3\pi/4,\pi)$.  We may make this assumption since otherwise we have $\pi/2 -\theta_t\ge \pi/4 \gg \la^{-1/4}$.
Thus, in the nontrivial case where $\alpha_t\notin (0,\pi/4)\cup (3\pi/4,\pi)$, since $K\le0$, we have for large enough $\la$
$$\text{Defect }Q\ge -\int_{B_r}K\, dV \ge \delta \la^{-N\e_0}=\delta \la^{-1/5},$$
for some $\delta>0$ by \eqref{g.1}.  Since we are assuming \eqref{g.2} we must have\linebreak $|\pi/2-\alpha_{s_0}|, \, |\pi/2-\alpha_{t_0}|\le \la^{-1/3}$ and
therefore
$$(\pi/2-\alpha_t)+(\pi/2-\alpha_s)\ge \delta\la^{-1/5}-2\la^{-1/3}\ge \tfrac\delta2 \la^{-1/5} \quad \text{if } \, \, \, \la \gg 1,$$
which of course implies that
\begin{equation}\label{g.7}
\max\bigl( \, \pi/2-\theta_t,\pi/2-\theta_s\, \bigr)\ge \la^{-1/4},
\end{equation}
if $\la$ is larger than a fixed constant which is independent of our two geodesic segments $\tilde \gamma_1$ and $\tilde \gamma_2$.

As noted before, the other case where $s\ne s_0$ and $|t-t_0|\ge \la^{-\e_0}$ is where the geodesics connecting
$\tilde \gamma_1(t_0)$ and $\tilde \gamma_2(s_0)$ and the one connecting $\tilde \gamma_1(t)$ and $\tilde \gamma_2(s)$ intersect at a point $P$.  Then
as in the second case Figure~\ref{fig2} we shall consider the geodesic triangle ${\cal T}$ with vertices $\tilde \gamma_1(t_0), \,\tilde \gamma_1(t)$ and $P$.  If
$\alpha_{t_0}, \, \alpha_t$ and $\alpha_P$ are the corresponding interior angles for ${\mathcal T}$, as before, we may assume that
$\alpha_t\notin (0,\pi/4)\cup(3\pi/4,\pi)$, for, if not, \eqref{g.7} trivially holds.   Then, as in Figure~\ref{fig2}, if $\la$ is large enough the geodesic
ball $B_r$, $r=\la^{-\e_0}/100$, which is tangent to $\tilde \gamma_1$ at $\tilde \gamma_1((t+t_0)/2)$ and on the same side as ${\mathcal T}$ must be contained
in ${\mathcal T}$ if $\la$ is larger than a fixed constant depending on the metric.  Therefore, by the Gauss-Bonnet theorem
$$\pi-(\alpha_{t_0}+\alpha_t+\alpha_P)=-\int_{\mathcal T} K\, dV\ge -\int_{B_r} K\, dV \ge \delta \la^{-N\e_0}=\delta \la^{-1/5}.$$
Therefore, by our assumption \eqref{g.2} and a variation of the earlier argument
$$\pi/2-(\alpha_t+\alpha_P) \ge \tfrac\delta2\la^{-1/5} \quad \text{if } \, \, \la\gg 1.$$
Since $\alpha_P>0$ this means that we must have
$$\theta_t=\alpha_t\in (0,\tfrac\pi2 -\tfrac\delta2 \la^{-1/5})\subset (0,\tfrac\pi2-\la^{-1/4}) \quad \text{if } \, \, \la\gg 1,$$
and so \eqref{g.7} is valid in this case as well.

The one remaining case to consider is where $s=s_0$ and $|t-t_0|\ge \la^{-\e_0}$.  One obtains \eqref{g.7} for this case as well by using
this argument but with ${\mathcal T}$ now being the geodesic triangle with vertices $\tilde \gamma_1(t_0)$, $\tilde \gamma_1(t)$ and $\tilde \gamma_2(s_0)$, which
completes the proof.\end{proof}


\begin{figure}[ht]

\makebox[\textwidth][c]{\includegraphics[width=1.2\textwidth]{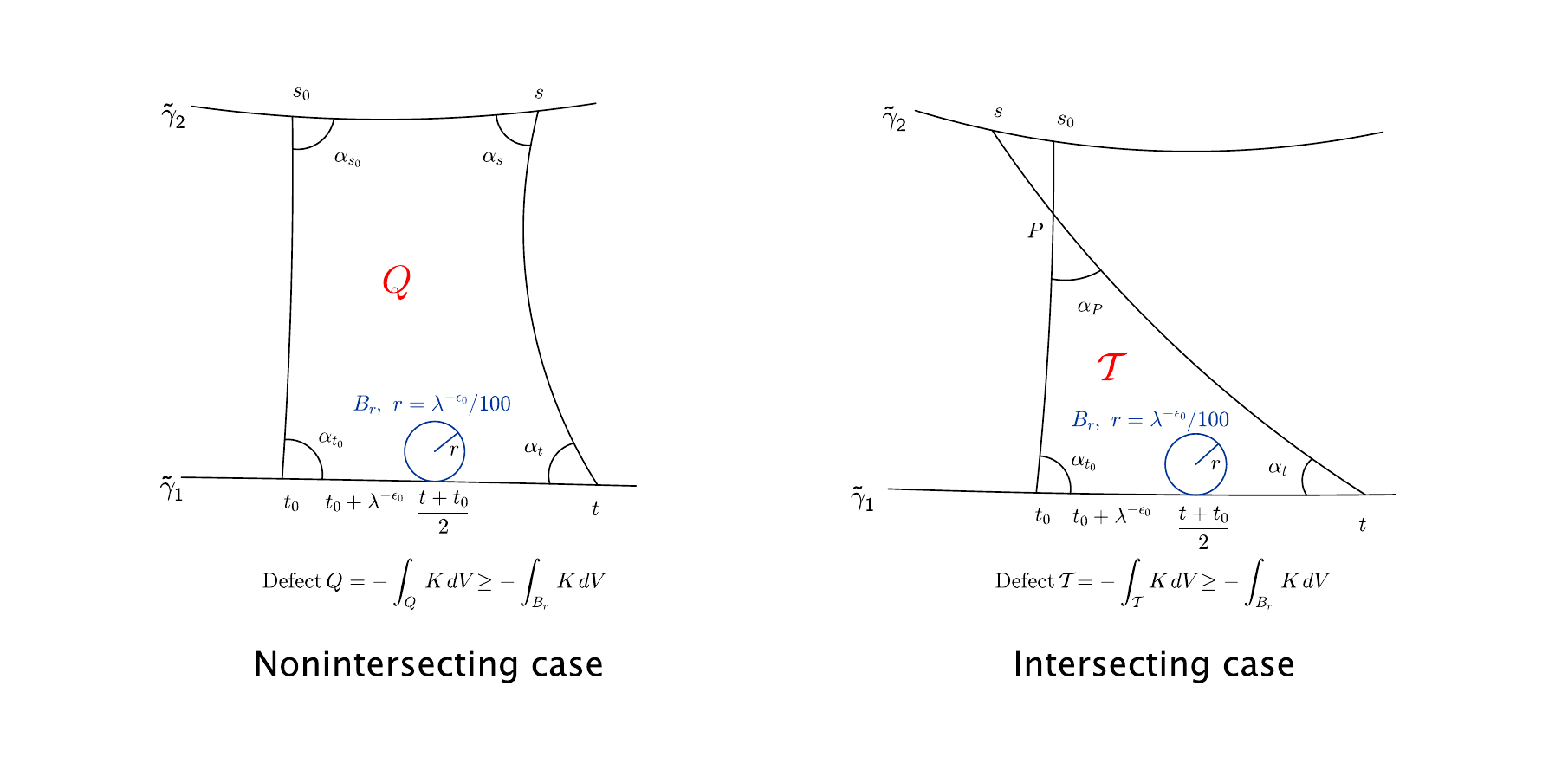}}
\caption{Two cases}
 \label{fig2}
\end{figure}

Even though Proposition~\ref{topprop} was proved in \cite{BSTop}, for the sake of completeness we shall give its simple proof now.

\begin{proof}[Proof of  Proposition~\ref{topprop}]
Recall that we are trying to show that
$$C(\theta_{T,r};T)\subset {\mathcal T}_r(\tilde \gamma), \quad \text{if } \, \sin \tfrac12\theta_{T,r} =\frac{\sinh \tfrac12 r}{\sinh T}.$$
We shall work in geodesic normal coordinates about $P_0$ and we may assume that, in these coordinates, $\tilde \gamma=
\{(t,0): \, t\in \R\}$.  $C(\theta; T)$ then is the intersection of the geodesic ball of radius $T>0$ about our origin
with the cone of aperture $\theta$ about $\tilde \gamma$.
Also,
${\mathcal T}_r(\tilde \gamma)$ denotes the closed tube of fixed radius $0<r<1$ about $\tilde \gamma$.

Since, for fixed $r$,  $T\to \theta_{T,r}$ is monotonically decreasing, it suffices to show that a point $Q$ with coordinates $T\omega$, $\omega\in S^{n-1}$,
belongs to ${\mathcal T}_r(\tilde \gamma)$ if the angle, $\sphericalangle (\omega,(1,0))$, is $\le \theta_{T,r}$.  In other words, to obtain
\eqref{g.5}, it suffices to show that
\begin{equation}\label{K6'} 
\Sigma(T;\theta_{T,r})\subset {\mathcal T}_r(\tilde \gamma),
\end{equation}
if $\Sigma(T,\theta)$ denotes all points $Q$ with coordinates $T\omega$ satisfying $\sphericalangle (\omega,\1)\le \theta$, with $\1 =(1,0)$.

Clearly $\Sigma(T;\theta)\subset {\mathcal T}_r(\tilde \gamma)$ when $\theta$ is very small (depending on $T$).  So choose the maximal
$\Theta_{T,r}\le \pi/2$ so that $\Sigma(T;\theta)\subset {\mathcal T}_r(\tilde \gamma)$ when $0<\theta<\Theta_{T,r}$.  It follows
that there must be a point $Q$ with coordinates,   $T\omega_0$, satisfying $\sphericalangle(\omega_0,\1)=\Theta_{T,r}$ and $d_{\tilde g}(Q,\tilde \gamma)=r$.
Also, \eqref{K6'} is valid when $\theta_{T,r}$ is replaced by $\Theta_{T,r}$.  So we would have \eqref{g.5} and be done if we could show that
\begin{equation}\label{K6''} 
\Theta_{T,r} \ge \theta_{T,r}.
\end{equation}

At this point, we shall use Toponogov's theorem.  First consider the geodesic triangle,
$\triangle^{\tilde g}_{\Theta_{T,r}}$, in $({\mathbb R}^2,\tilde g)$ with vertices $Q$
and the point with coordinates $0$ and the  point $P$ with coordinates $(T,0)$.  It is an isosceles triangle since
the geodesics connecting the point with coordinates $0$ with $P$ and $Q$, respectively, each have length $T$.  The point
$P$ lies on $\tilde \gamma$ and hence if $\tilde \gamma_{opp}$ is the third side of our geodesic triangle,
which connects $P$ and $Q$, we must have that its length, $\ell(\tilde \gamma_{opp})$ satisfies
 $$\ell(\tilde \gamma_{opp})=d_{\tilde g}(P,Q)\ge r,$$
 since, as we pointed out before, we must have  
$d_{\tilde g}(Q,\tilde \gamma)=r$.  The angle at the vertex whose coordinates are the origin, by construction, is $\Theta_{T,r}$,
and the two sides passing through it each have length $T$.
The third side of our isosceles triangle, $\tilde \gamma_{opp}$, is called a ``Rauch hinge''.

Consider as well, an isosceles triangle, $\triangle^{{\mathbb H}^2}_{\Theta_{T,r}}$, in two-dimensional hyperbolic space, ${\mathbb H}^2$, having two sides of equal
length $T$, angle $\Theta_{T,r}$ at the associated vertex and  ``Rauch hinge" $\gamma_{opp}$, with length
$\ell(\gamma_{opp})$.  By Toponogov's theorem (see \cite[Theorem 2.2 (B)]{ChEb}), since we are assuming that the
sectional curvatures of $({\mathbb R}^2,\tilde g)$ satisfy $-1\le K\le 0$, we must have
$$\ell(\gamma_{opp})\ge \ell(\tilde \gamma_{opp})\ge r.$$
By properties of isosceles triangles in ${\mathbb H}^2$, the ray bisecting the triangle at the vertex spanned by the two
sides of equal length $T$ must intersect the Rauch hinge, $\gamma_{opp}\in \triangle^{{\mathbb H}^2}_{\Theta_{T,r}}$, orthogonally at its midpoint.  Consequently,
by hyperbolic trigonometry, we must have
$$\sin\tfrac12 \Theta_{T,r} =\frac{\sinh(\ell(\gamma_{opp})/2)}{\sinh T}\ge \frac{\sinh \tfrac12 r}{\sinh T}= \sin\tfrac12 \theta_{T,r}.$$
Thus, \eqref{K6''} is valid and the proof of Proposition~\ref{topprop} is complete.
\end{proof}
\bigskip
 
 \newsection{Stationary phase bounds}
 
 Let us now collect the bounds for oscillatory integrals that we shall use to prove our bounds for smoothly localized integrals over geodesic segments and period integrals.
 These are more precise variations of the   ones used in the earlier work of Chen and the first author in \cite{CSPer}.
 
 The first concerns estimates  for one-dimensional oscillatory integrals with natural lower bounds for first derivatives of
 the phase function.

\begin{lemma}\label{lemmas.1}  Suppose that $\phi\in C^\infty((-1,1))$, is real valued
and that $a\in C^\infty_0({\mathcal I})$, where ${\mathcal I}\subset (-1,1)$ is 
an open interval and set
\begin{equation}\label{s.1}
I(\la)=\int e^{i\la \phi(t)} \, a(t)\, dt, \quad \la \ge1.
\end{equation}
Suppose that 
\begin{equation}\label{s.2}
|\partial_t \phi|\ge \la^{-1/2+\delta}, \quad \text{on  } \, {\mathcal I}
\end{equation}
and suppose 
further that for $0\le j\le N =\left \lceil{4\delta^{-1}}\right \rceil$.
\begin{equation}\label{s.3}
| \partial^j_t \phi' |\le  \lambda^{\delta/2}, \quad
\text{and } \, \, |\partial^j_t a|\le C_j \la^{j/2} \, \, \, \text{on } \, \, {\mathcal I}.
\end{equation}
Then if $0< \delta< 1/2$ 
\begin{equation}\label{s.4}
|I(\la)|\le C\la^{-2},
\end{equation}
where $C$ depends only on $\delta$ and the $C_j$.
\end{lemma}

As the following result says, we also can  obtain favorable estimates for one-dimensional oscillatory integrals if we do not have
the above hypothesis concerning lower bounds
for the  first derivatives of the phase, but rather have related lower  bounds for second derivatives.

\begin{lemma}\label{lemmas.2}
Set
\begin{equation}\label{s.5}
J(\la)=\int
e^{i\la \varphi(t)} \, b(t)\, dt, \quad \la\ge 1,
\end{equation}
where $b\in C^\infty_0({\mathcal I})$, where
${\mathcal I}$ is as above, and that $\varphi \in C^\infty( (-1,1))$ is real valued.
 Suppose further that $0\in {\mathcal I}$,
\begin{equation}\label{s.6}
|\varphi'(0)|\le \la^{-1/2+\delta}, \quad \text{and } \, \, \la^{-\delta/2}\le |\varphi''(t)|\le \la^{\delta/2},
\, \, \, t\in {\mathcal I},
\end{equation}
and that
\begin{equation}\label{s.7}
|b|\le 1, \quad |b'|\le \la^{1/2}.
\end{equation}
Then if $0<\delta\le 1/4$   there is a constant 
$C=C_\delta$ so that
\begin{equation}\label{s.8}
|J(\la)|\le C\la^{-1/2+2\delta}.
\end{equation}
\end{lemma}

\begin{proof}[Proof of Lemma~\ref{lemmas.1}]
Note that 
$$Le^{i\la \phi(t)}=e^{i\la\phi(t)}, \quad
\text{if } \, \, L=\frac1{i\la \phi'(t)}\frac{d}{d t}.
$$
Therefore,  if $L^*$ denotes the adjoint of $L$,  for every $N=1,2,3,\dots$,
\begin{equation}\label{s.9}
I(\la)=\int_{-\infty}^\infty e^{i\la \phi(t)} \bigl((L^*)^Na\bigr)(t) \, dt.\end{equation}
By a simple induction argument, one shows that $(L^*)^Na$ is a finite linear
combination of terms of the form
\begin{multline*}
\la^{-N}(\phi')^{-2N+j} \Bigl(\frac{d}{dt}\Bigr)^j a \cdot
\Bigl(\frac{d}{dt}\Bigr)^{\beta_1}\phi' \cdots \Bigl(\frac{d}{dt}\Bigr)^{\beta_K}\phi',
\\
\text{where } \quad j, \, \, \beta_\nu \in \{0,1,2,\dots, N\},
\quad \text{and } \, \, K\le N.
\end{multline*}
Therefore by \eqref{s.2} and \eqref{s.3} each of these terms is bounded by
$$
C_N \la^{-N}\la^{(2N-j)(1/2-\delta)}\la^{j/2} \la^{N\delta/2}
\le C_N\la^{-N\delta/2}.
$$
This and \eqref{s.9} gives us \eqref{s.4}.
%
\end{proof}

\begin{proof}[Proof of Lemma~\ref{lemmas.2}]
Fix $\rho\in C^\infty(\R)$ satisfying
\begin{equation}\label{s.10}
|\rho|\le 1, \, \, \rho(t)=1, \, \, |t|\le1, \, \, \text{and } \, \, \rho(t)=0, \, \, |t|\ge 2.
\end{equation}
Clearly
$$\Bigl | \int e^{i\la \varphi(t)} \, b(t) \, \rho(\la^{1/2-2\delta}t)
\, dt \, \Bigr|\le 4\la^{-1/2+2\delta},$$
and so it suffices to show that
$$\widetilde J(\la)=\int e^{i \la\varphi(t)} \, b(t) 
(1-\rho(\la^{1/2-2\delta}t))
\, dt$$
satisfies the bounds in \eqref{s.8}.

If we integrate by parts as in the proof of Lemma~\ref{lemmas.1} and use \eqref{s.7} and \eqref{s.10}, we see that
\begin{multline}\label{s.11}|\widetilde  J(\la)|
\le \la^{-1} \int_{\la^{-1/2+2\delta}\le |t|\le 1}
\Bigl(\, |\varphi'(t)|^{-1} \bigl| \, \tfrac{d}{dt}\bigl[b(t)(1-\rho(\la^{1/2-2\delta}t))\bigr]\bigr| 
\\ +
|b(t)|\, |\varphi'(t)|^{-2}|\varphi''(t)|\, \Bigr)\, dt.
\end{multline}

By the mean value theorem and \eqref{s.6}, for large enough $\la\ge1$, we have
\begin{equation*}
|\varphi'(t)|\ge \la^{-\delta/2}|t|-\la^{-1/2+\delta}\ge \frac12 \la^{-\delta/2}|t|,
\quad
\text{if } \, \,  t\in {\mathcal I} \cap \{t:  \la^{-1/2+2\delta}\le |t|\le 1\}.
\end{equation*}
Therefore, since we are assuming $\text{supp }b\subset{\mathcal I}$, by the second part of \eqref{s.6} and by \eqref{s.7}
\begin{align*}
| \widetilde J(\la)|&\lesssim \la^{-1} \int_{\la^{-1/2+2\delta}}^1\bigl(\, 
\la^{\delta/2} t^{-1} \bigl[\,  \la^{1/2} +t^{-1}\bigr] \, + \, \la^{3\delta/2} t^{-2}\, \bigr)\, dt
\\
&\lesssim \la^{-1/2+\delta/2}\ln \la + \la^{-1/2-3/2\delta} + \la^{-1/2-\delta/2} 
\\
&\lesssim \la^{-1/2+2\delta},
\end{align*}
as desired if $\la \gg 1$.
\end{proof}

We can combine Lemma~\ref{lemmas.1} and  Lemma~\ref{lemmas.2} to obtain the following.

\begin{proposition}\label{propositions.3}  Suppose that $\phi\in C^\infty(\R)$ is real valued and that $a\in C^\infty_0({\mathcal I})$, where ${\mathcal I}\subset (-1/2,1/2)$ is an open interval.  Suppose that
for some $0<\delta<1/2$
\begin{equation}\label{s.12}
\la^{-\delta/2}\le |\phi''(t)|, \quad t\in {\mathcal I}.
\end{equation}
Suppose further that for $0\le j\le N=\left \lceil{4\delta^{-1}}\right \rceil$
\begin{equation}\label{s.13}
|\partial^j_ta(t)|\le C_j\la^{j/2}
\end{equation}
and that
\begin{equation}\label{s.14}
|\partial^j_t \phi'|\le \la^{\delta/2}, \quad t\in{\mathcal I}.
\end{equation}
Then
\begin{equation}\label{s.15}
\Bigl|\, \int e^{i\la \phi(t)} \, a(t)\, dt \, \Bigr|\le C\la^{-1/2+2\delta},
\end{equation}
where $C$ depends only on $\delta$ and the above constants $C_j$, $j\le \left \lceil{4\delta^{-1}}\right \rceil$.
\end{proposition}


\begin{proof}To see this, we note that if \eqref{s.2} is valid we can replace \eqref{s.12} by the stronger bounds in \eqref{s.4}.  For the other case, where
\eqref{s.2} is not valid, there must be a point $t_0\in {\mathcal I}$ where $|\phi'(t_0)|\le \la^{-1/2+\delta}$.  We then get \eqref{s.15} from
\eqref{s.8} if we let the phase function $\varphi$ in Lemma~\ref{lemmas.2} be $\phi(t-t_0)$ and the bump function
$b\in C^\infty_0({\mathcal I}-\{t_0\})$ be $a(t-t_0)$, completing the proof.
\end{proof}

\newsection{Kernel bounds}

To be able to use the results from the last two sections to prove Proposition~\ref{prop2.1} and thus complete the proof of Theorem~\ref{mainthm}
we need to calculate the kernels in \eqref{2.16}, i.e, 
\begin{equation}\label{5.1}
K_{T,\la}(x,y)=\int \bigl(1-\beta(\tau)\bigr)\Hat \chi(\tau/T) e^{i\tau \la} \bigl(\cos \tau \Pe\bigr)(x,y) \, d\tau.
\end{equation}
Here, since all the calculations from now on will be taking place in the universal cover, to simplify the notation, we are setting $\Delta=\Delta_{\tilde g}$.  Also,
in what follows $\Delta_x^N$ and $\Delta_y^N$ will denote $N$ powers of the $\Delta_{\tilde g}$ with respect to the $x$ and $y$ variables,
respectively.

Recall that the bump function $\beta$ in \eqref{5.1} is supported in $(-4,4)$ and equals one on $[-3,3]$
and that $\Hat \chi(\tau)=0$ for $|\tau|\ge 1/2$.
Also recall that we are assuming, as in \eqref{2.2}, that $T=c\log \la$ where $c=c_M$ is a small positive
constant that will be specified later on.
Using this and the Hadamard parametrix
we shall obtain the following useful result.

\begin{proposition}\label{prop5.1}  If $d_{\tilde g}\ge1$ and $\la\gg 1$ we can write
\begin{equation}\label{5.2}
K_{T,\la}(x,y)= \la^{1/2}\sum_\pm a_\pm(T,\la; x,y) e^{\pm i\la d_{\tilde g}(x,y)} +R_{T,\la}(x,y),
\end{equation}
where
\begin{equation}\label{5.3}
|a_\pm(T,\la; x,y)|\le C,
\end{equation}
and
 if $\ell=1,2,3,\dots$ is fixed 
\begin{equation}\label{5.4}
\Delta^\ell_x a_\pm(T,\la; x,y) =O(\exp(C_\ell d_{\tilde g}(x,y)))
\end{equation}
or
\begin{equation}\label{5.5}
\Delta^\ell_y a_\pm(T,\la; x,y) =O(\exp(C_\ell d_{\tilde g}(x,y))),
\end{equation}
and
\begin{equation}\label{5.6}
|R_{T,\la}(x,y)|\le\la^{-1},
\end{equation}
provided that the constant $c>0$ in \eqref{2.2} is sufficiently small.
Also,  in this case we also have
\begin{equation}\label{5.7}
K_{T,\la}(x,y)=O(\la^{-1}), \quad \text{if } \, \, \, d_{\tilde g}(x,y)\le 1.
\end{equation}
\end{proposition}

Let us first handle the case were $d_{\tilde g}(x,y)\ge 1$ since proving \eqref{5.7} will be much easier than proving the 
first part of the Proposition.
Since $\cos \tau\Pe$ is self-adjoint, we only need to show that $K_{T,\la}$ can be written as in \eqref{5.2} where the 
amplitudes satisfy \eqref{5.5} and the remainder term is as in \eqref{5.6}.

To prove this we shall use the Hadamard parametrix as in B\'erard~\cite{Berard}.  As was shown there we can write for $|\tau|\ge1$
\begin{equation}\label{5.8}
\bigl(\cos \tau \Pe\bigr)(x,y)=\sum_{j=0}^m \alpha_j(x,y)\int_{-\infty}^\infty e^{i\theta(d^2-|\tau|^2)}|\tau| \, |\theta|^{1/2-j} \, d\theta +R(\tau,x,y),
\end{equation}
where $d=d_{\tilde g}(x,y)$,
\begin{equation}\label{5.9}
\alpha_0(x,y)=O(1) \quad \text{and } \, \, 
\alpha_j=O(\exp(C_j d)), \, \, j=1,2,\dots,
\end{equation}
and
\begin{equation}\label{5.10}
|\Delta_y^N\alpha_j|=O(\exp(C_Nd)), \quad j=0,1,\dots, \, \, N=1,2,\dots,
\end{equation}
and, if $m$ is large enough\footnote{Strictly speaking B\'erard~\cite{Berard} only stated this sort of bound
for $R$ itself in (42) on p. 263.  The proof of this particular pointwise estimate for the remainder was based
on energy estimates.  If one includes sufficiently many terms in \eqref{5.8} and uses higher order energy estimates 
one can obtain bounds like \eqref{5.11}.  (See also \cite{Hang}.)},
\begin{equation}\label{5.11}
|\partial^j_\tau R(\tau,x,y)| =O(\exp(Cd)), \quad j=0,1,2.
\end{equation}
In the above Fourier integrals we regularize the powers of $|\theta|$ near the origin at the expense of smooth errors that can be
absorbed in the remainder term.

The fact that the first coefficient, $\alpha_0$, in the Hadamard parametrix is bounded here is well known (see \cite{Hang}) and was
used, for instance, by the first author and Zelditch in the related work \cite{SZStein}.  It is a consequence of the G\"unther comparison theorem
and our assumption that $K\le 0$.   B\'erard~\cite{Berard} proved the other bounds \eqref{5.9}--\eqref{5.10} and used them, along with 
energy estimates, to obtain bounds of the form \eqref{5.11} for the remainder term in the parametrix.

If we change variables we can rewrite \eqref{5.8} in the more useful form
\begin{equation}\label{5.12}
\bigl(\cos \tau \Pe\bigr)(x,y)=\int_{-\infty}^\infty e^{i(d-|\tau|)\theta} q(\tau,x,y,\theta) \, d\theta +R(\tau,x,y),
\end{equation}
where the remainder term is as before and where
\begin{multline}\label{5.13}
\bigl|\partial^j_\tau \partial_\theta^k q(\tau,x,y,\theta)|\le C_{jk}\bigl[(1+|\theta|)^{1/2-k}(1+d+|\tau|)^{-j}
\\
+\exp(Cd)(1+|\theta|)^{-1/2-k}(1+d+|\tau|)^{-j}\bigr], \, \, \, \text{if } \, \, d+|\tau|\ge 1,
\end{multline}
as well as
\begin{multline}\label{5.14}
\bigl|\Delta^\ell_y \partial^j_\tau\partial_\theta^k q(\tau,x,y,\theta)\bigr|
\le C_{jk\ell}\exp(C_\ell d)(1+|\theta|)^{1/2-k}(1+d+|\tau|)^{-j}, 
\\
 \text{if } \,  d+|\tau|\ge 1.
\end{multline}

Since $\Hat \chi(\tau/T)=0$ if $|\tau|>T/2$, it is clear that by \eqref{5.11} and an integration by parts argument
$$\int\bigl(1-\beta(\tau)\bigr)\Hat \chi(\tau/T) e^{i\tau\la}R(\tau,x,y) \, d\tau =O(\la^{-2}\exp(CT)),
$$
and thus this term can be made to satisfy the bounds in \eqref{5.6} if $T=c\log\la$ with $c>0$ sufficiently small.

On account of this, if we plug the main term in \eqref{5.12} into \eqref{5.1}, we would have the first part of the proposition if we could show that
\begin{equation}\label{5.15}
\int_0^\infty \int_{-\infty}^\infty \bigl(1-\beta(\tau)\bigr) \Hat \chi(\tau/T) e^{i\tau \la} e^{i\theta(d-\tau)}q(\tau,x,y,\theta) \, d\theta d\tau
=\la^{1/2}e^{i\la d}a_+,
\end{equation}
and
\begin{equation}\label{5.16}
\int^0_{-\infty} \int_{-\infty}^\infty \bigl(1-\beta(\tau)\bigr) \Hat \chi(\tau/T) e^{i\tau \la} e^{i\theta(d+\tau)}q(\tau,x,y,\theta) \, d\theta d\tau
=\la^{1/2}e^{-i\la d}a_-,
\end{equation}
where $a_\pm$ satisfy the bounds in \eqref{5.3} and \eqref{5.4}.

To see this for \eqref{5.15} we note that the left side can be written as
$$\la^{1/2}e^{i\la d}\Bigl[\, \la^{-1/2}\int_0^\infty \int_{-\infty}^\infty \bigl(1-\beta(\tau)\bigr)\Hat \chi(\tau/T) e^{i(\theta-\la)(d-\tau)}
q(\tau,x,y,\theta) \, d\theta d\tau\, \Bigr].$$
Thus, if we set $a_+$ to be the term inside the square brackets, we can use \eqref{5.12} and integration by parts argument
to see that 
\begin{align*}
|a_+|&\le C\la^{-1/2}\iint (1+|\theta-\la|)^{-2}(1+|d-\tau|)^{-2}(1+|\theta|)^{1/2} \, d\tau d\theta
\\
&+C\exp(Cd)\la^{-1/2} \iint (1+|\theta-\la|)^{-2}(1+|d-\tau|)^{-2}(1+|\theta|)^{-1/2} \, d\tau d\theta
\\
&\le C\bigl(1+\exp(Cd) \la^{-1}\bigr).
\end{align*}
This yields the bounds in  \eqref{5.3} for $a_+$ if $T=c\log\la$ with $c>0$ small enough since $K(x,y)=0$ if $d_{\tilde g}(x,y)>T$.
If we repeat this argument and use \eqref{5.14} we also obtain the bounds in \eqref{5.4} for $a_+$ since
\begin{equation}\label{5.17}
\Delta^\ell_y d_{\tilde g}(x,y)=O(\exp(C_\ell d_{\tilde g}(x,y)))
\end{equation}
(which also follows from estimates in the appendix in B\'erard~\cite{Berard}).
Since the same argument shows
that \eqref{5.16} is valid with $a_-$ satisfying these two bounds, the proof of the first part of Proposition~\ref{prop5.1} is complete.

To prove \eqref{5.7} we recall that the factor $(1-\beta(\tau))=0$ if $|\tau|\le 3$ and so the bounds in \eqref{5.13} and \eqref{5.14} hold
on the support of the integrals in \eqref{5.15}--\eqref{5.16}.  Since $|d-\tau|\ge 1$ as well on the support if as in \eqref{5.7},
$d\le 1$, we conclude that \eqref{5.7} follows from a simple integration by parts argument. \qed

\bigskip

Note that \eqref{5.7} implies that the estimate in Proposition~\ref{prop2.1} is valid when $\alpha$ is the identity map.  To handle the
other nonzero summands in \eqref{2.16} we note that the kernel coming from the $\tau$-integration is
$K_{T,\la}(\tilde \gamma(t),\alpha(\tilde \gamma(s)))$ with $|t|$, $|s|\le 1/2$.  Our assumption that the injectivity radius of $M$ is ten or more
insures that $d_{\tilde g}(\tilde \gamma(t),\alpha(\tilde \gamma(s)))\ge 1$ in this case if $\alpha\ne Id$ and so we can
use \eqref{5.2}--\eqref{5.6}.

We shall need more information about the phase functions
\begin{equation}\label{5.18}
\phi(\alpha;t,s)=d_{\tilde g}(\tilde \gamma(t),\alpha(\tilde \gamma(s)))
\end{equation}
that arise from \eqref{5.2}.  Specifically, we shall require the following.

\begin{proposition}\label{prop5.2}  Let $\phi(\alpha;t,s)$ be as in \eqref{5.18} with $\alpha\ne Id$.  Then for each $j=1,2,3,\dots$ there is a constant $C_j$ so that
\begin{multline}\label{5.19}
|\partial_t^j\phi(\alpha;t,s)|+|\partial_s^j\phi(\alpha;t,s)|\le \exp(C_jT),
\\ \text{if } \, \, \max\{d_{\tilde g}(\tilde \gamma(t),\alpha(\tilde \gamma(s))): \, \, |t|, \, |s|\le 1/2\} \le T.
\end{multline}
Moreover, we have the uniform bounds
\begin{equation}\label{5.20}
|\partial_t\partial_s\phi(\alpha;t,s)|\le C.
\end{equation}
Additionally,
\begin{equation}\label{5.21}
\partial^2_t\phi(\alpha;t_0,s_0) \ge \exp(-CT), 
\quad \text{if } \, \, d_{\tilde g}(\tilde \gamma(t_0),\alpha(\tilde \gamma(s_0)))\le T
\end{equation}
for some $C$ if
\begin{equation}\label{5.22}
|\partial_t\phi(\alpha;t_0,s_0)|\le 1/4.
\end{equation}
\end{proposition}

The bound in \eqref{5.19} for $\partial_t^j\phi$ follows from \eqref{5.17}.  Since $\phi(\alpha;t,s)=d_{\tilde g}(\tilde \gamma(s),\alpha^{-1}(\tilde \gamma(t)))$,
the bound for $\partial_s^j\phi$ also follows from \eqref{5.17}.

To prove \eqref{5.20} we may work in geodesic normal coordinates about $\tilde \gamma(t)$, with $\tilde \gamma$ being the first coordinate axis in these
coordinates.  Write $\alpha(\tilde \gamma(s))=(x_1(s),x_2(s))$ in these
coordinates.  Then
$$\frac{\partial\phi}{\partial t}(\alpha;t,s)=-\frac{x_1(s)}{|x(s)|}.$$
Thus
$$\frac{\partial^2\phi}{\partial t\partial s}=
\frac{x_1(s)\bigl[\tfrac{d}{ds}(x_1^2(s)+x_2^2(s))\bigr]}{|x(s)|^2} -\frac{\dot x_1(s)}{|x(s)|}=O(|\dot x(s)|)$$
since our assumptions give $|x(s)|\ge 1$.  Consequently we would have \eqref{5.20} if we could show that
\begin{equation}\label{5.23}
|\dot x(s)|=O(1).
\end{equation}

To do this we note that since $\alpha$ is an isometric mapping $\alpha(\tilde \gamma(s))=x(s)$ must be a geodesic.  We recall that
if
$$p(x,\xi)=\sqrt{\sum_{j,k=1}^2 g^{jk}(x)\xi_j\xi_k}$$
where $g^{jk}$ is the cometric, i.e., $g^{jk}=(g_{jk})^{-1}$, then by the Hamilton-Jacobi formulation of unit speed geodesic
flow (see, e.g. \S 2.3 in \cite{Hang}) we must have that
$$\dot x(s) =\frac{\partial p}{\partial \xi}\bigl(x(s),\xi(s)\bigr),
\quad \text{where } \, \, p(x(s),\xi(s))=1.$$

Therefore, to get \eqref{5.23} it suffices to show that in our geodesic coordinate system
\begin{equation}\label{5.24}
\bigl|\sum_{j=1}^2 g^{jk}(x)\xi_j\bigr|\le 1\quad \text{if } \, \, \sum_{j,k=1}^2 g^{jk}(x)\xi_j\xi_k=1.
\end{equation}
After a rotation $U$
$$U^t g^{jk}(x) U=\text{diag}(1,|g|^{-1}),$$
where $|g|=\det (g_{jk}(x))$.  By volume comparison estimates since we are assuming that $K\le 0$ in $(M,g)$ and hence in $(\Rt,\tilde g)$ we
must have that $|g|\ge1$ since we are working in geodesic normal coordinates about $\gamma(t)$. 

If $\xi=U\eta$  then
\begin{equation}\label{5.25}
\sum_{j,k=1}^2 g^{jk}(x)\xi_j\xi_k=1 \, \, \iff \, \, \eta_1^2+|g|^{-1}\eta_2^2=1.
\end{equation}
Thus, since $|g|\ge1$
\begin{multline*}
\bigl|\bigr(g^{jk}(x)\bigr)\xi\bigr|^2=\bigl|\bigr(g^{jk}(x)\bigr)U\eta\bigr|^2 =\bigl|U^t g^{jk}(x)U\eta\bigr|^2
\\
=\eta_1^2+|g|^{-2}\eta_2^2\le \eta_1^2+|g|^{-1}\eta_2^2\le 1,
\end{multline*}
as desired, which completes the proof of \eqref{5.23} and \eqref{5.20}.

To prove \eqref{5.21} we shall again work in geodesic normal coordinates, this time about $\tilde \gamma(t_0)$, 
again with $\tilde \gamma$ being the first coordinate axis.  Then, as before
$$\frac{\partial \phi}{\partial t}(\alpha;t_0,s_0)=-\frac{x_1(s)}{|x(s)|}=-\cos \theta_{s_0}(t_0),$$
where $\theta_{s_0}(t_0)\in [0,\pi)$ denotes the intersection angle of the geodesic ray $\tilde \gamma(t)$, $t\ge t_0$ with the
geodesic ray starting at $\tilde \gamma(t_0)$ and passing through $x(s_0)=\alpha(\tilde \gamma(s_0))$.  See Figure~\ref{fig4}.
For $\Delta t>0$ small, as in this figure, consider the angle $\theta_{s_0}(t_0+\Delta t)$ formed by the geodesic
ray $\tilde \gamma(t)$, $t>t_0+\Delta t$ and the geodesic ray passing through $\tilde \gamma(t_0+\Delta t)$ and
$\alpha(\tilde \gamma(s_0))$ as in the Figure.  Then
$$\frac{\partial^2 \phi}{\partial t^2}(\alpha;t_0,s_0) = \sin(\theta_{s_0}(t_0)) \lim_{\Delta t\searrow 0}
\frac{\theta_{s_0}(t_0+\Delta t)-\theta_{s_0}(t_0)}{\Delta t}.$$
Our assumption \eqref{5.22} means that $\sin \theta_{s_0}(t_0)\ge 1/10$ and so
$$ \frac{\partial^2 \phi}{\partial t^2}(\alpha;t_0,s_0)  \approx \lim_{\Delta t\searrow 0}\, 
 \frac{\theta_{s_0}(t_0+\Delta t)-\theta_{s_0}(t_0)}{\Delta t} .$$
Since $K\le 0$, by the Rauch comparison theorem (see \cite{Kling}), if $\Delta \theta$ denotes the angle of the aforementioned rays
through $\alpha(\tilde \gamma(s_0))$ as in Figure~\ref{fig4}, then we must have
$$\Delta \theta +\theta_{s_0}(t_0)+\bigl(\pi -\theta_{s_0}(t_0+\Delta t)\bigr)\le \pi,$$
since $\pi-\theta_{s_0}(t_0+\Delta t)$, $\theta_{s_0}(t_0)$ and $\Delta \theta$ are the three interior angles for the triangle
with vertices $\tilde \gamma(t_0)$, $\tilde \gamma(t_0+\Delta t)$ and $\alpha(\tilde \gamma(s_0))$.
Thus,
$$ \frac{\partial^2 \phi}{\partial t^2}(\alpha;t_0,s_0)  \approx \lim_{\Delta t\searrow 0} \frac{\theta_{s_0}(t_0+\Delta t)-\theta_{s_0}(t_0)}{\Delta t}
\ge \lim_{\Delta t\searrow 0}\frac{\Delta \theta}{\Delta t}.$$
By Proposition~\ref{topprop}, we must have that
$\Delta \theta \ge \sinh (\Delta t/2)/\sinh T$, which leads to \eqref{5.21} and completes the proof of Proposition~\ref{prop5.2}.  \qed

\begin{figure}[ht]
\centering
\includegraphics[width=1.1\textwidth]{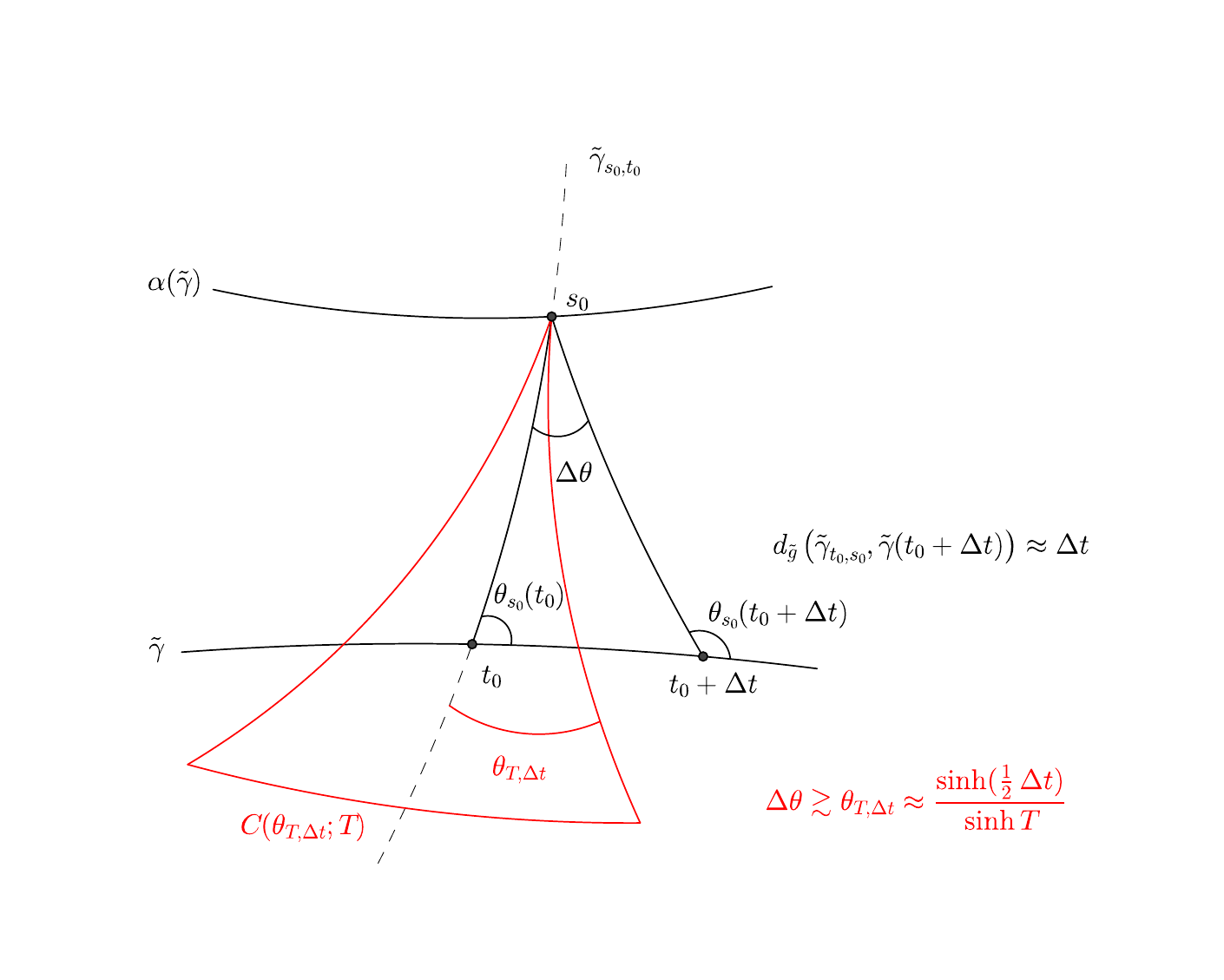}
\caption{Bounding Hessian from below}
 \label{fig4}
\end{figure}

\bigskip

We also need the following simple consequence of Proposition~\ref{propg.1}, which was based on our assumption \eqref{i.2}.

\begin{proposition}\label{prop5.3}  Let $\phi(\alpha;t,s)$  be is as above with $\alpha \ne Id$.   Then if for $t_0,s_0\in [-1/2,1/2]$ and  $\la\gg 1$
\begin{equation}\label{5.26}
\bigl|\nabla_{t,s}\phi(\alpha;t_0,s_0)\bigr| \le \frac12 \la^{-1/3},
\end{equation}
it follows that if $\e_0$ is as in \eqref{g.3}, we have
\begin{multline}\label{5.27}
\bigl|\nabla_{t,s}\phi(\alpha;t,s)\bigr|\ge \frac12 \la^{-1/4}, 
\\
 \text{if } \, \, \, 
\max\bigl(|t-t_0|, \, |s-s_0|\bigr)\ge \la^{-\e_0} \, \, \text{and } \, |t|,|s|\le 1/2.
\end{multline}
\end{proposition}

By the above arguments $|\partial_t\phi(\alpha;t,s)|=\cos \theta_t$ and $|\partial_s \phi(\alpha;t,s)|=\cos \theta_s$ where $\theta_t$ and
$\theta_s$ are as in Proposition~\ref{propg.1}.  From this one immediately sees that Proposition~\ref{prop5.3} follows from 
Proposition~\ref{propg.1}.  \qed

\newsection{End of proof of period integral estimates}

In this section we shall complete the proof of Proposition~\ref{prop2.1} and hence that of Theorem~\ref{mainthm}.  We need to verify that we can fix
$c=c_M>0$ so that for $\la\gg1$ \eqref{2.16} is valid for some $\delta_M>0$.  We shall take 
$$\delta_M=\e_0/2$$
where $0<\e_0<1/10$ is as in \eqref{g.3} and \eqref{5.27}.

As we pointed out earlier, we know that \eqref{2.16} is valid when $\alpha= Id$.  Hence it suffices to show that all the other nonzero terms
there satisfy
\begin{equation}\label{6.1}
\Bigl|\, \iint b(t,s) K_{T,\la}\bigl(\tilde \gamma(t),\alpha(\tilde \gamma(s))\bigr) \, dt ds \, \Bigr|
\le C_{b,M}\la^{-\e_0/2},
\end{equation}
if $T=c_M\! \log\la$ with $c_M>0$  small enough.  Recall that $b\in C^\infty_0((-1/2,1/2)^2)$.

In view of the estimate \eqref{5.6} for the remainder term in \eqref{5.2}, it suffices to show that
\begin{equation}\label{6.2}
\la^{1/2}\Bigl| \, \iint b(t,s)a_\pm(T,\la; \tilde \gamma(t), \alpha(\tilde \gamma(s))) \, e^{\pm i \la\phi(\alpha;t,s)}\, dt ds\, \Bigr|
\le C_{b,M}\la^{-\e_0/2}
\end{equation}
under the above assumptions with $\phi(\alpha;t,s)=d_{\tilde g}(\tilde \gamma(t),\alpha(\tilde \gamma(s)))$ as in 
Propositions \ref{prop5.2} and \ref{prop5.3}.  As  noted before, we have $\phi(\alpha; t,s)=d_{\tilde g}(\alpha^{-1}(\tilde \gamma(t)), \tilde \gamma(s))$.  
Also since $\bigl(\cos \tau \sqrt{-\Delta_{\tilde g}}\bigr)(x,y)=\bigl(\cos \tau \sqrt{-\Delta_{\tilde g}}\bigr)(\alpha^{-1}(x),\alpha^{-1}(y))$, we have that 
$K_{T,\la}(x,y)=K_{T,\la}(\alpha^{-1}(x),\alpha^{-1}(y))$ and so
\begin{equation*}
a_\pm(T,\la; \tilde \gamma(t),\alpha(\tilde \gamma(s)))=a_\pm(T,\la,\alpha^{-1}(\tilde \gamma(t)),\tilde \gamma(s)).
\end{equation*}

Therefore, by \eqref{5.4}, \eqref{5.5} and \eqref{5.17} if $T=c\log\la$ with $c=c_M>0$ sufficiently small and if $K_{T,\la}$ does not
vanish identically we have
\begin{multline}\label{6.3}
|\partial^j_t\phi(\alpha;s,t)|+|\partial^j_s\phi(\alpha;s,t)|+|\partial_t^ja(t,\la; \tilde \gamma(t),\alpha(\tilde \gamma(s)))|
\\
+|\partial_s^ja(t,\la; \tilde \gamma(t),\alpha(\tilde \gamma(s)))|
\le \la^{\e_0/8}, \quad \text{with } \, \, \, 1 \le j\le \left \lceil{8\e_0^{-1}}\right \rceil.
\end{multline}
We use $\e_0/8$ here since we shall eventually want to apply Proposition~\ref{propositions.3} with $\delta=\e_0/4$.

To apply Proposition~\ref{prop5.3} and the stationary phase bounds from \S 4, we shall consider two cases:
\begin{equation}\label{6.4}
|\nabla_{t,s}\phi(\alpha;s,t)|\ge \frac12 \la^{-1/3}, \quad |t|, \, |s| <1/2,
\end{equation}
and the complementary case where
\begin{equation}\label{6.5}
|\nabla_{t,s}\phi(\alpha;t_0,s_0)|\le \frac12 \la^{-1/3} \quad \text{for some } \, \, (t_0,s_0)\in (-1/2)\times (-1/2).
\end{equation}

To show that \eqref{6.2} is valid under the assumption \eqref{6.4} we shall use a partition of unity argument to exploit \eqref{6.3}.  Specifically,
choose $\rho\in C^\infty_0((-1,1))$ satisfying 
$$\sum_{j=-\infty}^\infty \rho(t-j)\equiv 1, \quad t\in \R.$$
Then for $m=(m_1,m_2)\in {\mathbb Z}^2$ set
$$\rho_m(t,s)=\rho(\la^{1/2}t-m_1) \, \rho(\la^{1/2}s-m_2).$$
It follows that $\sum_{m\in {\mathbb Z}^2}\rho_m(t,s)\equiv 1$ and that $|\partial_t^j\rho_m| +|\partial_s^j\rho_m|\le C_j\la^{j/2}$.  Also, $\rho_m$
is supported in a $O(\la^{-1/2})$ size neighborhood about $(t_m,s_m)=(\la^{-1/2}m_1,\la^{-1/2}m_2)$.  Assuming that this neighborhood intersects
\linebreak
$(-1/2,1/2)\times (-1/2,1/2)$ and that $(\overline{t_m},\overline{s_m})$ is in the intersection, by \eqref{6.4} we must have that
$$|\partial_t\phi(\alpha;\overline{t_m},\overline{s_m})|\ge \frac14\la^{-1/3}, 
\, \, \, \text{or } \, \, |\partial_s\phi(\alpha;\overline{t_m},\overline{s_m})|\ge \frac14\la^{-1/3}.$$
Let us assume the former since the argument for the latter is similar.  By \eqref{5.20} and \eqref{6.3} we must have that
$$|\partial_t\phi|\ge \frac14\la^{-1/3}-O(\la^{-1/2}\la^{\e_0/4})-O(\la^{-1/2})\ge \frac18\la^{-1/3} \quad
\text{on  supp }\rho_m$$
for large $\la$ since $\e_0<1/10$.  Therefore, by Lemma~\ref{lemmas.1}, we have that
$$\la^{1/2}\Bigl| \, \int b(t,s)\rho_m(t,s)a_\pm(T,\la; \tilde \gamma(t), \alpha(\tilde \gamma(s))) \, e^{\pm i\la \phi(\alpha;t,s)} \, dt \, \Bigr|
\le C_{b,\e_0}\la^{-3/2}.$$
Since $\rho_m(t,s)=0$ if $|s-s_m|\ge C\la^{-1/2}$, this in turn gives the bounds
$$\la^{1/2}\Bigl| \, \iint b(t,s)\rho_m(t,s)a_\pm(T,\la; \tilde \gamma(t), \alpha(\tilde \gamma(s))) \, e^{\pm i\la \phi(\alpha;t,s)} \, dtds \, \Bigr|
\le C_{b,\e_0}\la^{-2}.$$
Since there are $O(\la)$ such terms which are nonzero, we conclude that when \eqref{6.4} holds we obtain a stronger version of 
\eqref{6.2} where $\la^{-\e_0/2}$ is replaced by $\la^{-1}$.

To complete the proof, we must show that \eqref{6.2} is valid when we assume \eqref{6.5}.  We shall use
Proposition~\ref{prop5.3} for this (which makes use of our curvature assumption \eqref{i.2}).  To this end, let $\beta\in C^\infty_0(\R)$ be as
above, i.e.,
$$\beta(t)=1 \, \, \text{on } \, \, [-3,3] \, \, \, \text{and } \, \, \, \beta(t)=0 \, \, \, \text{for } \, \, |t|\ge4.$$
We then obtain
\begin{multline*}
\Bigl|\, \iint \bigl(1-\beta(\la^{\e_0}|t-t_0|)\beta(\la^{\e_0}|s-s_0|)\bigr) \, a_\pm(T,\la; \tilde \gamma(t), \alpha(\tilde \gamma(s)))
\, e^{\pm i\la\phi(\alpha;t,s)} \, dt ds \, \Bigr|
\\
\le C_{b,\e_0}\la^{-3/2}
\end{multline*}
by the previous argument since by \eqref{5.27} 
$$|\nabla_{t,s}\phi(\alpha;t,s)|\ge \frac12\la^{-1/4} \quad 
\text{if } \, \, \bigl(1-\beta(\la^{\e_0}|t-t_0|)\beta(\la^{\e_0}|s-s_0|)\bigr)\ne 0.$$

Thus, our proof would be complete if we could show that
\begin{multline}\label{6.6}
\Bigl|  \iint  \beta(\la^{\e_0}|t-t_0|)\beta(\la^{\e_0}|s-s_0|)  b(t,s) \, a_\pm(T,\la; \tilde \gamma(t), \alpha(\tilde \gamma(s))) \,
e^{\pm i \la\phi(\alpha:t,s)}  dt ds \Bigr| 
\\
\le C_{b,M}\la^{-\e_0/2-1/2}.
\end{multline}
To do this, we note that, by \eqref{5.21}, we have that when $T=c\log\la$ with $c=c_M>0$ small,
$$|\partial_t^2\phi(\alpha;t,s)|\ge \la^{-\e_0/8} \quad \text{if } \, \, 
\beta(\la^{\e_0}|t-t_0|)\beta(\la^{\e_0}|s-s_0|)\ne 0,$$
since our assumption \eqref{6.5} along with \eqref{6.3} and \eqref{5.20} ensure that \eqref{5.22} is valid for such $(t,s)$.
Therefore, by Proposition~\ref{propositions.3} with $\delta=\e_0/4$, we have that for each $|s|\le 1/2$
\begin{multline*}
\la^{1/2}\Bigl|  \int  \beta(\la^{\e_0}|t-t_0|)\beta(\la^{\e_0}|s-s_0|)  b(t,s) \, a_\pm(T,\la; \tilde \gamma(t), \alpha(\tilde \gamma(s))) 
e^{\pm i \la \phi(\alpha:t,s)}  dt   \Bigr| 
\\
\le C_{b,M}\la^{\e_0/2}.
\end{multline*}
Since $s\to \beta(\la^{\e_0}|s-s_0|)$ is supported in an interval of size $\approx \la^{-\e_0}$ this implies \eqref{6.6}, which completes
the proof.  \qed

\bibliography{EF}{}
\bibliographystyle{amsplain}

\end{document}